\documentclass[11pt]{amsart}

\usepackage[pdftex]{graphicx} \usepackage[dvipsnames]{xcolor}

\usepackage[margin=2cm]{geometry}

\usepackage{amsfonts}
\usepackage{amsmath}
\usepackage{amssymb}
\usepackage{amsthm}
\usepackage{mathtools}
\usepackage{tikz}
\usetikzlibrary{patterns}

\usepackage{paralist}
\usepackage[colorlinks,cite color=blue,pagebackref=true,pdftex]{hyperref}

\usepackage[T1]{fontenc}

\usepackage{caption}
\usepackage{subcaption}

\newcommand\Defn[1]{\textbf{\color{black}#1}}
\newcommand\Def[1]{\Defn{#1}}

\renewcommand\emptyset{\varnothing}
\newcommand\Z{\mathbb{Z}}

\newcommand\R{\mathbb{R}}
\newcommand\Q{\mathbb{Q}}

\newcommand\inner[1]{\langle {#1} \rangle}
\newcommand\defeq{\coloneqq}

\DeclareMathOperator{\relint}{relint}

\DeclareMathOperator{\conv}{conv}
\DeclareMathOperator{\cone}{cone}

\DeclareFontFamily{U}{mathx}{\hyphenchar\font45}
\DeclareFontShape{U}{mathx}{m}{n}{ <5> <6> <7> <8> <9> <10> <10.95> <12> <14.4>
  <17.28> <20.74> <24.88> mathx10 }{}
\DeclareSymbolFont{mathx}{U}{mathx}{m}{n}
\DeclareFontSubstitution{U}{mathx}{m}{n}
\DeclareMathAccent{\widecheck}{0}{mathx}{"71}

\newtheorem{thm}{Theorem} \newtheorem{cor}[thm]{Corollary}
\newtheorem{lem}[thm]{Lemma} \newtheorem{prop}[thm]{Proposition}
 \newtheorem{quest}{Question}

\theoremstyle{definition}

\title{Coprime Ehrhart theory and counting free segments}

\author{Sebastian Manecke} 
\author{Raman Sanyal}

\address{Institut f\"ur Mathematik, Goethe-Universit\"at Frankfurt, Germany} 
\email{manecke@math.uni-frankfurt.de}
\email{sanyal@math.uni-frankfurt.de}



\date{\today}

\begin{document}

\begin{abstract}
A lattice polytope is \emph{free} (or \emph{empty}) if its vertices are the
only lattice points it contains. In the context of valuation theory, Klain
(1999) proposed to study the functions $\alpha_i(P;n)$ that count the number
of free polytopes in $nP$ with $i$ vertices. For $i=1$, this is the famous
Ehrhart polynomial. For $i > 3$, the computation is likely impossible and for
$i=2,3$ computationally challenging.

In this paper, we develop a theory of coprime Ehrhart functions, that count
lattice points with relatively prime coordinates, and use it to compute
$\alpha_2(P;n)$ for unimodular simplices. We show that the coprime Ehrhart
function can be explicitly determined from the Ehrhart polynomial and we give
some applications to combinatorial counting.
\end{abstract}
\maketitle

\newcommand\TypeCone{\mathcal{T}}%
\newcommand\InSpc{\mathcal{I}_+}%

\section{Introduction}\label{sec:intro}

In this paper, we are exclusively concerned with \Def{lattice polytopes}, that
is, convex polytopes $P$ with vertices in $\Z^d$, for some $d$. A lattice
polytope $P$ is called \Def{free} (or \Def{empty}) if $P \cap \Z^d$ are
precisely the vertices of $P$.  Free polytopes have been studied in relation
to integer programming; see, for example,~\cite{Reznick,Reeve, Scarf,
HaaseZiegler,santos} and they are related to \emph{hollow} polytopes, whose
(relative) interior do not contain lattice points. Our interest in free
polytopes comes from valuation theory and geometric combinatorics. For a set
$S \subset \R^d$, denote by $[S] : \R^d \to \{0,1\}$ its indicator function
and let $\alpha_1(S) \defeq | S \cap \Z^d |$. Klain~\cite{Klain} basically
proved the following identity for lattice polytopes $P$:
\begin{equation}\label{eqn:ind}
    (-1)^{\dim P} [\relint(P)] \ = \ - \sum_{Q} (-1)^{\alpha_1(Q)} [Q] \, ,
\end{equation}
where the sum is over all free polytopes $Q \subseteq P$.
Applying the Euler characteristic to both sides of~\eqref{eqn:ind} then yields
\[
    1 \ = \ - \sum_{Q} (-1)^{\alpha_1(Q)} \ = \ \sum_{i \ge 1} (-1)^i
    \alpha_i(P) \, ,
\]
where we set $\alpha_i(P)$ to be the number of free polytopes $Q \subseteq P$
with $\alpha_1(Q) = i$. Klain~\cite[Sect.~10]{Klain} proposed to study the
functions
\[
    \alpha_i(P,n) \ \defeq \ \alpha_i(n \cdot P)
\]
where $n \cdot P = \{ np : p \in P \}$ is the $n$-th integer dilate of $P$.
This is motivated by the fact $\alpha_1(P,n) = | n \cdot P \cap \Z^d |$ is the
famous Ehrhart polynomial; see, for example,~\cite{CRT}. For $i > 1$, the
function $\alpha_i$ is not a valuation on polytopes. Moreover, in dimensions
$\ge 3$, there are infinitely many free polytopes up to unimodular equivalence
on a fixed number of $i \ge 4$ vertices, which probably renders task of
computing $\alpha_i(P;n)$ hopeless in general. However, any two free segments
are unimodularly equivalent and there is hope for the computation of
$\alpha_2(P;n)$. Here, $\alpha_2(P)$ is the number of pairs in $P \cap \Z^d$
that are \emph{visible} from each other. If $P$ is a dilate of the unit cube,
then this is related to \emph{digital lines}~\cite{Dorst}. For the unit square
the sequence $\alpha_2([0,1]^2;n)$ is given in~\cite{cube} but an explicit
description does not seem to be known.

\newcommand\CE{\mathrm{CE}}%
The main goal of our endeavor was to find explicit description of
$\alpha_2(P;n)$ where $P$ is a unimodular simplex; see Theorem~\ref{thm:alpha_2}
below. In our investigation, it turned out that we need the following
number-theoretic variation of Ehrhart theory: The \Def{coprime Ehrhart
function} of a lattice polytope $P \subset \R^d$ is the function
\begin{equation}\label{eqn:CE}
    \CE(P;n) \ \defeq \ |\{ (a_1,\dots,a_d) \in nP \cap \Z^d :
    \gcd(a_1,\dots,a_d,n) = 1 \}| \, .
\end{equation}
If $P$ is a \emph{half-open} free segment (that is, one endpoint removed),
then $\CE(P;n) = \phi(n)$. Theorem~\ref{thm:CE} below gives an explicitly
computable description for general lattice polytopes.

Recall that the \Def{Ehrhart function} of $P$ is $ E(P;n)  \defeq  | nP \cap
\Z^d | $ for $n \in \Z_{\ge 0}$. Ehrhart~\cite{ehrhart} famously proved that
the Ehrhart function agrees with polynomial of degree $r = \dim P$: there are
numbers $e_i(P) \in \R$ for $i=0,\dots,r$ such that
\begin{equation}\label{eqn:Ehr}
    E(P;n) \ = \ e_r(P) n^r + e_{r-1}(P) n^{r-1} + \cdots + e_0(P) n^0 \quad
    \text{ for all } n \ge 0 \, .
\end{equation}

For $k \ge 0$, the \Def{Jordan totient function} is given by
\[
    J_k(n) \ \defeq \ \left |\Bigl\{ (a_1,\dots,a_k) \in \Z^k : 
    \begin{array}{c}
        1 \le a_i \le n \text{ for } i \\
        \gcd(a_1,\dots,a_k,n) = 1 
    \end{array}
    \Bigr\}\right | \, .
\]
For $k=0$, we have $J_0(n) = 1$ if $n=1$ and $=0$ otherwise.  For $k = 1$,
$J_1(n) = \phi(n)$ is the Euler totient function. See~\cite{numth} for more on
properties of $J_k(n)$.

\begin{thm}\label{thm:CE}
    Let $P$ be an $r$-dimensional lattice polytope. Then
    \[
        \CE(P;n) \ = \ e_r(P) J_r(n) + e_{r-1}(P) J_{r-1}(n) + \cdots + e_0(P)
        J_0(n) \, ,
    \]
    for all $n \ge 0$.
\end{thm}

The Jordan totient function can be computed as
\[
    J_k(n) \ = \ n^k \prod_{p \mid n} \left( 1 - \frac{1}{p^k} \right) \, ,
\]
where $p$ ranges over all prime factors of $n$. This prompts us to define
$J_k(-n) \defeq (-1)^k J_k(n)$ for all $n \ge 0$. The following is the counterpart
to Ehrhart--Macdonald reciprocity (see~\cite{CRT}), which states that the
number of lattice points in the relative interior of $nP$ is given by
$(-1)^{\dim P} E(P;-n)$.
\begin{thm}\label{thm:CE-rec}
    Let $P \subset \R^d$ be an $r$-dimensional lattice polytope and $n \ge 1$.
    Then
    \[
        \CE(\relint(P);n) \ = \ (-1)^r \CE(P;-n) \, .
    \]
\end{thm}

\newcommand\SL{\mathrm{SL}}%
Theorem~\ref{thm:CE} allows us to give an easily computable description of
$\alpha_2(P;n)$, where $P$ is a unimodular simplex. Let us first note that
$\alpha_2(P;n)$ is invariant under \Def{unimodular transformations}, that is,
linear transformations $T(x) = Ax + b$ with $A \in \SL(\Z^d)$ and $b \in
\Z^d$. It is therefore sufficient restrict to the $d$-dimensional
\Def{standard simplex} 
\[
    \Delta_d \ \defeq \ \conv(e_1, \dots, e_{d+1}) \ = \  \{ x \in \R^{d+1} :
    x \ge 0, x_1 + \cdots + x_{d+1} = 1 \} \, .
\]
We also define the polytope $\nabla_d \defeq \Delta_d + (-\Delta_d)$, the
difference body~\cite[Sect.~10.1]{Schneider} of $\Delta_d$.

\begin{thm}\label{thm:alpha_2}
    For $d \ge 1$, we have
    \[
        \alpha_2(\Delta_d;n) \ = \ \frac{1}{2} \sum_{\ell=0}^n
        \binom{n-\ell+d}{d} \CE(\partial \nabla_d;\ell) \, .
    \]
\end{thm}

The Ehrhart polynomial of $\nabla_d$ is given in~\eqref{eqn:nabla-ehr} in
Section~\ref{sec:alpha_2}. Together with Ehrhart--Macdonald reciprocity this
gives $E(\partial \nabla_d;n) = E(\nabla_d;n) - (-1)^d E(\nabla_d;-n)$ and we
can use Theorem~\ref{thm:CE} to compute $\alpha_2(\Delta_d;n)$. For $2 \le d
\le 9$ this yields the following list:
\begin{equation}
\begin{aligned}
  \CE(\partial \nabla_2;n) \ &= \ 6J_1(n)\\
  \CE(\partial \nabla_3;n) \ &= \ 10J_2(n) + 2J_0(n)\\
  \CE(\partial \nabla_4;n) \ &= \ \tfrac{35}{3}J_3(n) + \tfrac{25}{3}J_1(n)\\
  \CE(\partial \nabla_5;n) \ &= \ \tfrac{21}{2}J_4(n) + \tfrac{35}{2}J_2(n) + 2J_0(n)\\
  \CE(\partial \nabla_6;n) \ &= \ \tfrac{77}{10}J_5(n) + \tfrac{49}{2}J_3(n) + \tfrac{49}{5}J_1(n)\\
  \CE(\partial \nabla_7;n) \ &= \ \tfrac{143}{30}J_6(n) + \tfrac{77}{3}J_4(n) + \tfrac{707}{30}J_2(n) + 2J_0(n)\\
  \CE(\partial \nabla_8;n) \ &= \ \tfrac{143}{56}J_7(n) + \tfrac{429}{20}J_5(n) + \tfrac{297}{8}J_3(n) + \tfrac{761}{70}J_1(n)\\
  \CE(\partial \nabla_9;n) \ &= \ \tfrac{2431}{2016}J_8(n) + \tfrac{715}{48}J_6(n) + \tfrac{4147}{96}J_4(n) + \tfrac{14465}{504}J_2(n) + 2J_0(n)\\
\end{aligned}
\end{equation}

In the plane, there are exactly four free polytopes up to unimodular
equivalence. In particular, there are no free polygons with more than four
vertices. Using the relations given in~\cite[Cor.~7.4]{Klain}, this allows us
to completely determine all functions $\alpha_i(P;n)$ for unimodular
triangles.

\begin{cor}\label{cor:plane}
    Let $P \subset \R^2$ be a unimodular triangle.
    \begin{align*}
        \alpha_1(P;n) \ &= \ \binom{n+1}{2} \, , &
        \alpha_3(P;n) \ &= \ 6 \sum_{\ell=0}^n
        \binom{n-\ell+2}{2} J_1(\ell) -2n^2 - 3n
        \\
        \alpha_2(P;n) \ &= \ 3 \sum_{\ell=0}^n
        \binom{n-\ell+2}{2} J_1(\ell) \, , &
        \alpha_4(P;n) \ &= \ \sum_{\ell=0}^n
        \binom{n-\ell+2}{2} J_1(\ell) - \frac{3}{2} (n^2 + n)
    \end{align*}
    and $\alpha_i(P) = 0$ for $i > 4$.
\end{cor}

The paper is organized as follows. In Section~\ref{sec:CE}, we briefly develop
coprime Ehrhart theory and prove Theorems~\ref{thm:CE} and~\ref{thm:CE-rec}.
Section~\ref{sec:alpha_2} is devoted to the study of $\alpha_2(\Delta_d;n)$.
We close with afterthoughts and open questions in Section~\ref{sec:thoughts}.

\textbf{Acknowledgements.} This project grew out of discussions in the seminar
\emph{Point configurations, valuations, and anti-matroids} at the Goethe
University Frankfurt in the (unusual) summer 2020. We thank the participants
for creating a wonderful albeit virtual atmosphere.
We also thank Lionel Pournin for pointing out~\cite{DP}.
Our research was driven by experiments conducted with
\textsc{Sage}~\cite{SAGE} and
\emph{The On-Line Encyclopedia of Integer Sequences}~\cite{OEIS}.

\section{Coprime Ehrhart functions}\label{sec:CE}

A \Def{valuation} on lattice polytopes is a function $\varphi$ satisfying
$\varphi(\emptyset) = 0$ and 
\[
    \varphi(P \cup Q) \ = \ \varphi(P) + \varphi(Q) - \varphi(P \cap Q) 
\]
for all lattice polytopes $P,Q$ such that $P \cup Q$ and $P \cap Q$ are also
lattice polytopes. We call $\varphi$ \Def{lattice-invariant} if $\varphi(T(P))
= \varphi(P)$ for all unimodular transformations $T(x)$.

\begin{lem}
    Let $n \ge 1$ be fixed. Then the map $P \mapsto \CE(P;n)$ is a
    lattice-invariant valuation.
\end{lem}
\begin{proof}
    It is straightforward to verify that $P \mapsto \CE(P;n)$ is a valuation.
    To see lattice invariance, let $a \in n \cdot T(P) = A(nP) + nb$. That is
    $a = A a' + nb$ for some lattice point $a' \in nP$. It is now clear that
    $\gcd(a,n) = \gcd(Aa' + nb,n) = \gcd(Aa',n) = \gcd(a',n)$. Hence
    $\CE(T(P);n) = \CE(P;n)$.
\end{proof}

\begin{proof}[Proof of Theorem~\ref{thm:CE}]
    In their seminal paper Betke and Kneser~\cite{BetkeKneser} show that every
    lattice-invariant valuation is uniquely determined by its values on the
    unimodular simplices $S_k = \conv(0,e_1,\dots,e_k)$ for $k=0,\dots,d$.
    This implies that the vector space of real-valued and lattice-invariant
    valuations has dimension $d+1$. It is straightforward to verify that for
    every $i =0,\dots,d$, the function $P \mapsto e_i(P)$ is a
    lattice-invariant valuation. Since $e_i(P)$ is homogeneous of degree $i$,
    this implies that the valuations $\{ e_i(P) : i=0,\dots,d \}$ constitute a
    basis for the space of real-valued and lattice-invariant valuations.
    Thus, for $n \ge 1$ fixed, there are $c_{n,i} \in \R$ such that 
    \begin{equation}\label{eqn:CE-coeff}
        \CE(P;n) \ = \ c_{n,d} e_d(P) + c_{n,d-1} e_{d-1}(P) + \cdots +
        c_{n,0} e_{0}(P)
    \end{equation}
    To determine the coefficients $c_{n,i}$ it suffices
    evaluate~\eqref{eqn:CE-coeff} at sufficiently many lattice polytopes or,
    in fact, \emph{half open} polytopes; see the methods used in~\cite{JS}.
    The $k$-dimensional \Def{half-open} cube is $H_k \defeq (0,1]^k$. Its Ehrhart
    polynomial is readily given by $E(H_k;n) = n^k$ and hence $e_j(H_k) = 1$
    if $j=k$ and $=0$ otherwise. To complete the proof, we simply note that
    $\CE(H_k;n) = J_k(n)$.
\end{proof}

\begin{proof}[Proof of Theorem~\ref{thm:CE-rec}]
    In order to prove Theorem~\ref{thm:CE-rec}, we recall the following
    implication of a classical result due to McMullen~\cite{mcmullenEuler}. If
    $\varphi$ is a lattice-invariant valuation and $P$ an $r$-dimensional
    lattice polytope, then the function $\varphi_P(k) \defeq \varphi(kP)$ agrees
    with a polynomial of degree at most $r$. Moreover $(-1)^{r} \varphi_P(-1)
    = \varphi(\relint(-P))$.

    If we set $\varphi(P) \defeq \CE(P;n)$ for $n \ge 1$ fixed, then we obtain
    \[
        \varphi_P(k) \ = \ e_r(P) J_r(n) k^r + e_{r-1} J_{r-1}(n) k^{r-1} +
        \cdots + e_0(P) J_0(n) k^0 
    \]
    and hence
    \[
        \CE(\relint(P);n) \ = \ \CE(\relint(-P);n) \ = \ (-1)^r \varphi_P(-1)
        \ = \ (-1)^r \CE(P;-n)
        \qedhere
    \]
\end{proof}

To conclude this section, let us briefly remark that both results can also be
proved with the use of number-theoretic M\"obius inversion. For this, we note
that
\[
    E(P;n) \ = \ \sum_{d \mid n} \CE(P;\tfrac{n}{d})
\]
and hence
\[
    \CE(P;n) \ = \ \sum_{d \mid n} \mu(d) E(P;\tfrac{n}{d}) \, .
\]
using linearity, we only need to consider
\[
\sum_{d \mid n} \mu(d) \left(\frac{n}{d}\right)^k \ =
n^k \sum_{d \mid n} \mu(d) \frac{1}{d^k}
\ = \ n^k \prod_{p \mid n} \left( 1- \frac{1}{p^k} \right) \ = \ J_k(n) \, .
\]

The Jordan totient functions $J_k(n)$ take the role of the monomial basis
$n^k$. In Ehrhart theory and combinatorics, there are two more important
bases. For $d \ge 0$, let 
\begin{align*}
    S_d \ &= \ \{ x \in \R^d : x \ge 0, x_1 + \cdots + x_d \le 1 \}\, \text{
        and} \\
    O_d \ &= \ \{ x \in \R^d : 0 \le x_1 \le x_2 \le \cdots \le x_d \le 1\}\,.
\end{align*}
Both polytopes are unimodular simplices with Ehrhart polynomial $E(S_d;n) =
E(O_d;n) = \binom{n+d}{d}$. Ehrhart--Macdonald reciprocity now states
$\binom{n-1}{d} = E(\relint(O_d);n) = (-1)^d E(O_d,-n) = (-1)^k
\binom{-n+d}{d}$. In particular
\[
    E(S_d;n) \ = \ \sum_{k=0}^d \frac{c(n,k)}{n!} n^k \, ,
\]
where $c(n,k)$ are the unsigned Stirling numbers of the first
kind~\cite[Prop.~1.3.7]{EC1}.  This formula was also discovered in connection
with primitive point packings by Deza and Pournin~\cite{DP}.

The corresponding coprime Ehrhart function has a
nice interpretation.
\begin{prop}
    For $k \ge 1$, $B_k(n) \defeq \CE(S_k;n)$ is the number of compositions
    $\mu = (\mu_1,\dots,\mu_l)$ with $\mu_i \ge 1$ and $\mu_1+ \cdots + \mu_l
    = n$ of length $l \le k+1$ with $\gcd(\mu_1,\dots,\mu_l) = 1$.
\end{prop}

The function $(-1)^k B_k(-(n+1))$ was studied by Gould~\cite{Gould} under the
name $R_k(n)$ as the number of compositions of $n$ with exactly $k$ relatively
prime parts. The functions $B_k(n)$ and $R_k(n)$ take the role of the binomial
coefficients in Coprime Ehrhart theory. 

We can also consider the fraction of lattice points in $nP$ that get counted
by $\CE(P;n)$ as $n$ goes to infinity. 
\begin{cor}
    Let $P \subset \R^d$ be an $r$-dimensional lattice polytope. Then
    \[
        \limsup_{n \to \infty} \frac{\CE(P;n)}{E(P;n)} \ = \
        \frac{1}{\zeta(r)} \, ,
    \]
    where $\zeta$ is the Riemann zeta function.
\end{cor}
\begin{proof}
    \[
        \limsup_{n \to \infty} \frac{\CE(P;n)}{E(P;n)} \ = \
        \limsup_{n \to \infty} \frac{J_r(n)}{n^r} \ = \
        \limsup_{n \to \infty} \prod_{p \mid n} \left(1-\frac{1}{p^r}\right)
        \ = \ \frac{1}{\zeta(r)} \, . \qedhere
    \]
\end{proof}
In the case that $P$ is the unit cube, this seems to be related to~\cite{EW}.

\section{Counting free segments in unimodular simplices}
\label{sec:alpha_2}

\newcommand\Prim{\mathcal{P}}%

In this section, we will determine $\alpha_2(P;n)$, the number of free
segments contained in $nP$, where $P$ is a unimodular $d$-simplex. We start by
some considerations that apply to general lattice polytopes.  

Let $P \subset \R^d$ be a lattice polytope and $S = [a,b]$ a free segment.
For $n \ge 1$, we write
\begin{equation}\label{eqn:gEhr}
    E_S(P;n) \ = \ |\{ t \in \Z^d : t + S \subseteq nP\}| \ = \ | (nP -a) \cap
    (nP - b) \cap \Z^d | \, .
\end{equation}
Note that $E_S(P;n)$ is invariant under translation of $S$ and we may assume
that $a = 0$. We call a vector $b \in \Z^d$ \Def{primitive} if $\gcd(b) = 1$
and we write $E_b(P;n) = E_{[0,b]}(P;n)$.  This gives us the representation
\begin{equation}\label{eqn:alpha2-general}
    \alpha_2(P;n) \ = \ \frac{1}{2} \sum_{b \text{ primitive}} E_{[0,b]}(P;n)
    \, .
\end{equation}
The factor $\frac{1}{2}$ stems from the fact that $[0,-b] = [0,b] - b$.

The functions $E_S(P;n)$ are vector partition functions~\cite{sturmfels} and
related to \emph{multivariate} Ehrhart functions. If $P$ is a unimodular
simplex, then the next result states that $E_S(P;n)$ is in fact an
Ehrhart polynomial. 

We now consider the standard unimodular simplex $\Delta_d \subset \R^{d+1}$
and define
\[
    \Prim^d \ \defeq \ \{b \in \Z^{d+1} : \gcd(b_1, \dots, b_{d+1}) = 1, b_1 +
  \dots + b_{d+1} = 0 \} \, .
\]
For $b \in \Prim^d$, there are unique $b^+,b^- \in \Z^d_{\ge 0}$ such that $b
= b^+ - b^-$ and $b^+_i b^-_i = 0$ for $i=1,\dots,d$. We further define
\[
    \ell(b) \ \defeq \ \sum_i b^+_i \ = \ \sum_i b^-_i  \ .
\]

\begin{lem}\label{lem:E_S}
    Let $b \in \Prim^d$ and $n \ge 1$. Then $n\Delta_d \cap (n\Delta_d -b) =
    \emptyset$ if and only if $\ell(b) > n$. If $\ell(b) \le n$, then 
    \[
        n\Delta_d \cap (n\Delta_d -b) \ = \ b^- + (n-\ell(b)) \Delta_d \, .
    \]
    In particular, $E_b(n) = \binom{n-\ell(b) + d}{d}$ for $\ell(b) \le n$ and
    $E_b(n) = 0$ otherwise.
\end{lem}
\begin{proof}
    The polytope $n\Delta_d \cap (n\Delta_d - b)$ is given by all points $x
    \in \R^{d+1}$ such that
    \[
        x_1 + \cdots + x_{d+1} = n \quad \text{ and } \quad  x_i \ge
        \min(0,-b_i) \text{ for all } i=1,\dots,d+1 \, .
    \]
    Summing the inequalities yields $x_1+ \cdots + x_{d+1} \ge \ell(b)$. Thus,
    there is no solution if and only if $\ell(b) > n$. Otherwise, the
    solutions are given by points of the form $x = b^- + x'$ with $\sum_i x'_i
    = n-\ell(b)$ and $x' \ge 0$. This shows our second claim, namely,
    $n\Delta_d \cap (n\Delta_d -b)  = b^- + (n-\ell(b)) \Delta_d$. 
    The final statement follows from the fact that $E_b(\Delta_d;n) = E(
    \Delta_d; n-\ell(b) ) = \binom{n-\ell(b)+d}{d}$. 
\end{proof}

Combining Lemma~\ref{lem:E_S} with~\eqref{eqn:alpha2-general}, yields
\[
    \alpha_2(\Delta_d;n) \ = \ \frac{1}{2}\sum_{b \in \Prim^d}
    E_{b}(\Delta_d;n) \ = \
    \frac{1}{2} \sum_{\substack{p \in \Prim^d\\\ell(b) \leq n}} 
    \binom{n-\ell(b)+d}{d} \ = \
    \frac{1}{2} \sum_{\ell=0}^n \binom{n-\ell+d}{d} c_{\ell,d}\,,
\]
where $c_{\ell,d} \defeq \# \{b \in \Prim^d : \ell(S) = \ell\}$. 

Let $\nabla_d \defeq \Delta_d + (-\Delta_d)$. This is a convex polytope with
vertices $e_i - e_j$ for $i\neq j$. The combinatorial and arithmetic structure
of $\nabla_d$ is easy to understand; see, for example,~\cite[Sect.~3]{CFS} and
below.

\begin{prop}
    Let $d, \ell \ge 1$. Then
    \[
        c_{\ell,d} \ = \ \CE(\partial( \Delta_d - \Delta_d ); \ell) \, .
    \]
\end{prop}
\begin{proof}
    Let $b \in \Z^{d+1}$ with $\sum_i b_i = 0$ and recall that $b = b^+ - b^-$
    where $b^+,b^- \in \Z^{d+1}_{\ge0}$ with disjoint supports. It follows
    that $b \in  \partial(k \nabla_d)$ if and only if $\ell(b) = k$. Adding
    the condition $\gcd(b,n) = 1$ now proves the claim.
\end{proof}

For $c \in \R^{d+1}$, let $\nabla_d^c \defeq \{ x \in \nabla_d : \inner{c,x}
\ge \inner{c,y} \text{ for all } y \in \nabla_d \}$ be the face maximizing the
linear function $x \mapsto \inner{c,x}$.  Let $I_+ \defeq \{ i \in [d+1] : c_i
= \max(c) \}$ and $I_- \defeq \{ i \in [d+1] : c_i = \min(c) \}$. Then
\[
    \nabla_d^c \ = \ \conv( e_i - e_j : i \in I_+, j \in I_- ) \, .
\]
If $c$ is not a multiple of $(1,\dots,1)$, then $I_+ \cap I_- = \emptyset$ and
$\nabla_d^c$ is unimodularly equivalent to $\Delta_{|I_+|-1} \times
\Delta_{|I_-|-1}$. The number of faces that are isomorphic to $\Delta_{k-1}
\times \Delta_{l-1}$ is $\binom{d+1}{k,l}$. The boundary of $\nabla_d$ is the
disjoint union of the relative interiors of proper faces. This gives us the
following expression
\begin{equation}\label{eqn-c_l,d}
    c_{\ell,d} \ = \ 
    \sum_{\substack{k,l \ge 1\\ k + l \leq d+1}} \binom{d+1}{k,l}
    \CE(\relint(\Delta_{k-1}\times\Delta_{l-1}); \ell) \, .
\end{equation}

Note that $\relint(\Delta_{k-1}\times\Delta_{l-1}) = \relint(\Delta_{k-1})
\times \relint(\Delta_{l-1})$. Hence
\[
    E(\relint(\Delta_{k-1}\times\Delta_{l-1});n) \ = \
    E(\relint(\Delta_{k-1});n) \cdot E(\relint(\Delta_{l-1});n) \ = \
    \binom{n-1}{k-1} \binom{n-1}{l-1} 
\]
and using Theorem~\ref{thm:CE} yields explicit expressions for $c_{\ell,d}$
and subsequently for $\alpha_2(\Delta_d;n)$.

A different representation of $c_{\ell,d}$ is as follows. The polytope
$\nabla_d$ is \Def{reflexive} (cf.~\cite{Batyrev}), that is, its \emph{polar}
dual is again a lattice polytope. This implies that
$E(\relint(\nabla_d);n) = E(\nabla_d;n-1)$ and hence $E(\partial \nabla_d;n) =
E(\nabla_d;n) - E(\nabla_d;n-1)$. Using~\cite[Cor.~3.16]{CFS}, an explicit
description of $E(\nabla_d;n)$ is given by
\begin{equation}\label{eqn:nabla-ehr}
    E(\nabla_d;n) \ = \ \sum_{j=0}^d \binom{d}{j} \binom{n}{j}
    \binom{n+d-j}{d-j} \, .
\end{equation}

\section{Afterthoughts and questions}\label{sec:thoughts}

\subsection{Geometric combinatorics and coprime chromatic functions}

There are a number of counting functions that can be expressed in terms of
Ehrhart polynomials; see~\cite{CRT}. Perhaps most prominent is the chromatic
polynomial of a graph. Let $G = (V,E)$ be a simple graph. An
\Def{$n$-coloring} is a map $c : V \to \{1,\dots,n\}$ with $c(u)\neq c(v)$ for
all edges $uv \in E$. Birkhoff~\cite{birkhoffcoloring} introduced the function
$\chi_G(n)$ counting the number of $n$-colorings of a graph. Birkhoff and,
independently, Whitney~\cite{whitneychromatic} proved that $\chi_G(n)$ agrees
with a polynomial in $n$ of degree $d = |V|$. This is known as the
\Def{chromatic polynomial} of $G$. Beck and Zaslavsky~\cite{iop} realized
$\chi_G(n)$ as an Ehrhart polynomial of an \emph{inside-out} polytope. An
explicit formula is given by
\[
    \chi_G(n) \ = \ \sum_{k=0}^d w_k(G) n^{d-k} \, ,
\]
where $w_k(G)$ are the Whitney numbers of the first kind of
$G$;~\cite{Aigner}. 

Now, a \Def{coprime coloring} is an $n$-coloring $c$ with the additional
constraint that the set $c(V) \cup \{n\}$ is coprime. If we denote by
$\chi^c_G(n)$ the number of coprime $n$-colorings, then Theorem~\ref{thm:CE}
readily gives us
\[
    \chi^c_G(n) \ = \ \sum_{k=0}^d w_k(G) J_{d-k}(n) \, .
\]
Similarly, we may define \emph{coprime order functions} on posets;
see~\cite[Ch.~1]{CRT}.

\subsection{Rational coprime Ehrhart theory and coprime $\Pi$-partitions}

If $P \subset \R^d$ is a polytope with vertices in $\Q^d$, then $E(P;n)$
agrees with a \emph{quasi-polynomial}. That is, there are periodic functions
$c_i(n)$ such that
\[
E(P;n) \ = \ c_d(n) n^d + \cdots + c_0(n) n^0 \quad \text{ for all } n \ge 1\, .
\]

\begin{quest}
    Can the coprime Ehrhart function $\CE(P;n)$ of a \emph{rational} polytope
    be related to its Ehrhart function?
\end{quest}

Our construction of coprime Ehrhart functions is in line with the usual
approach to Ehrhart theory. For a rational polytope $P \subset \R^d$, its
\Def{homogenization} is the pointed polyhedral cone $C(P) = \{ (x,t) : t \ge
0, x \in tP \} = \cone( P \times \{1\})$. The set $M(P) = C(P) \cap \Z^{d+1}$
is a finitely generated monoid and $E(P;n) = |\{ (a,t) \in M(P) : t = n\}|$.
If we denote by $\Z^{d+1}_{\mathrm{prim}} \defeq \{ a \in \Z^{d+1}, \gcd(a) =
1\}$, then $\CE(P;n) = |\{ (a,n) \in \Z^{d+1}_{\mathrm{prim}} : (a,n) \in M(P)
\}|$. A \Def{grading} of $M(P)$ is a linear function $\ell : \Z^{d+1} \to \Z$
such that $\ell(p) > 0$ for all $p \in M(P) \setminus 0$. The associated
Hilbert function $H_\ell(P;n) = |\{ p \in M(P) : \ell(p) = n \}|$ is a
quasipolynomial for all rational polytopes $P$~\cite[Sect.~4.7]{CRT}. This
rests on the rationality of the \Def{integer point transform} $\sum_{p \in M(P)}
z^p \in \Z[\!\![ z_1^{\pm1}, \dots, z_{d+1}^{\pm1} ]\!\!]$.

\begin{quest}
    Is there a coprime version of the rational integer point transform?
\end{quest}

A nice combinatorial consequence would be a coprime theory of
$\Pi$-partitions. Let $\Pi$ be a finite set partially ordered by $\preceq$. A
$\Pi$-partition of $n \ge 0$ is a map $f : \Pi \to \Z_{\ge0}$ such that
$\sum_{a \in \Pi} f(a) = n$ and $f(a) \le f(b)$ whenever $a \preceq b$. This
setup was introduced by Stanley (see~\cite{EC1}) as a generalization of usual
partitions and plane partitions. It can be shown that the function $c_\Pi(n)$
counting the $\Pi$-partitions of $n$ is of the form $H_\ell(P;n)$ for some
rational polytope $P$ and linear function $\ell$. A $\Pi$-partition is
\Def{strict} if $f(a) > 0$ and $f(a) < f(b)$ when $a \prec b$. It would be
desirable to obtain explicit formulas for counting coprime (strict)
$\Pi$-partitions.

\newcommand\rpp{\mathrm{rpp}}%
Work in this direction was done by El Bachraoui~\cite{Bachraoui}. 
A \Def{relatively prime partition} of $n \in \Z_{\ge0}$ is a sequence of
natural numbers $\lambda_1 > \lambda_2 > \cdots > \lambda_k > 0$ such that $n
= \lambda_1 +\lambda_2 + \cdots + \lambda_k$ and the $\lambda_i$ are coprime.
The number of parts of the partition is $k$. We write $\rpp_k(n)$ for the
number of relatively prime partitions of $n$ with exactly $k$ parts. Note that
$\rpp_2(n)$ is the number of coprime $0 < a < b$ with $n = a+b$ and hence
$\rpp_2(n) = \frac{1}{2} \varphi(n)$. For the number of relatively prime
partitions with $3$ parts El Bachraoui~\cite{Bachraoui} showed that for $n \ge
4$
\[
    \rpp_3(n) \ = \ \frac{1}{12} J_2(n) \, .
\]

\subsection{Mixed versions}

Upon closer inspection of the proof of Theorem~\ref{thm:alpha_2}, it can be
seen that 
\[
    \alpha_2(\Delta_d;n) \ = \ |\{ (t,b) \in (\Z^{d+1})^2 : b \in n\nabla_d, t
    \in b^+ + (n-\ell(b)) \Delta_d, \gcd(b,n) = 1 \}| \,.
\]
This prompts the question of a \emph{mixed} version of $\CE(P;n)$. For a
lattice polytope $P \subset \R^d$ and $I \subseteq [d]$, define
\[
    \CE_I(P;n) \ := \ |\{ p \in nP \cap \Z^d : \gcd(\{ p_i : i \in I\} \cup
    \{n\}) = 1 \}| \, .
\]
It would be interesting if a reasonable expression for $\CE_I(P;n)$ could be
found in general.

\subsection{Counting free triangles}

In every dimension $\ge 2$, the unimodular triangle is the unique free
polytope with $3$ vertices, up to unimodular equivalence. 

\begin{quest}
    Is there a closed expression for $\alpha_3(\Delta_d;n)$?
\end{quest}

\bibliographystyle{siam} \bibliography{bibliography}

\end{document}